\documentclass[12pt]{amsart}
\usepackage{amsmath,amssymb,amsthm,mathtools}
\usepackage{mathrsfs,bm}
\usepackage{thmtools,thm-restate}
\usepackage{tikz-cd}
\usepackage{tikz}
\usepackage{enumitem}
\usepackage{graphicx}
\usetikzlibrary{matrix,arrows.meta,bending}
\usepackage{color}
\usepackage[numbers]{natbib}
\usepackage{booktabs}
\usepackage{longtable}
\usepackage{pgfplots}
\pgfplotsset{compat=1.14} 
\usetikzlibrary{arrows.meta,decorations.pathreplacing,positioning,calc} 

\usepackage[a4paper, left=3cm, right=3cm, top=3cm, bottom=3cm, footskip=15mm]{geometry}

\newtheorem{theorem}{Theorem}[section]

\newtheorem{proposition}[theorem]{Proposition}

\newtheorem{conjecture}[theorem]{Conjecture}
\theoremstyle{definition}
\newtheorem{definition}[theorem]{Definition}

\theoremstyle{remark}

\numberwithin{equation}{section}

\makeindex

\usepackage{fancyhdr}
\usepackage{calc} 

\AtBeginDocument{%
  \setlength{\textwidth}{\dimexpr\paperwidth - 6cm\relax}%
  \setlength{\oddsidemargin}{\dimexpr 3cm - 1in\relax}%
  \setlength{\evensidemargin}{\dimexpr 3cm - 1in\relax}%
  \setlength{\marginparwidth}{2.5cm}%
}

\begin{document}

\title{On the Gap sequence and the Gilbreath conjecture}

\author{T. Agama}
\address{Department of Mathematics, African Institute for Mathematical science, Ghana
}
\email{theophilus@aims.edu.gh/emperordagama@yahoo.com}


\date{\today}

\dedicatory{}


\begin{abstract}
Motivated by the Gilbreath conjecture, we develop the notion of the gap sequence induced by any sequence of numbers. We introduce the notion of the path and associated circuits induced by an originator and study the conjecture via the notion of the trace and length of a path.
\end{abstract}

\maketitle

\begingroup
  \setlength{\parskip}{6pt} 
  \tableofcontents
\endgroup

\section{Introduction}

Gilbreath's conjecture occupies a distinctive place in the theory of prime gaps because it converts the irregularity of the primes into a rigid recursive pattern of first entries under repeated differencing. Starting from the ordered primes $p_1,p_2,\dots$, one forms the first gap sequence $d_n^{1}=p_{n+1}-p_n$ and then iterates the unsigned forward difference operator by setting $d_n^{k}=|d_{n+1}^{k-1}-d_n^{k-1}|$ for $k\ge 2$. The conjecture asserts that the leading entry of every row remains equal to $1$. This phenomenon was already observed in the nineteenth-century work of Proth and later reappeared independently in Gilbreath's note; modern discussion of the conjecture emphasizes both its elementary formulation and its surprisingly deep resistance to proof. The best-known computational verification is due to Odlyzko, who confirmed the conjecture far beyond the range accessible to hand calculations, while the more recent work of Chase shows that the same differencing paradigm admits a rigorous probabilistic analog for sufficiently random sequences with controlled gaps \cite{proth1878theoremes,odlyzko1993iterated,guy2004unsolved,chase2024random}.\\

The purpose of this paper is not merely to restate the conjecture in different notation but to build a framework in which the recursive differencing process becomes a structured combinatorial object. The central idea is to interpret the iterated gap table as a family of finite paths generated by an originator sequence. Each row of the difference triangle is viewed as a path of some order, each term in the row as a segment, and the repeated shortening of rows as the natural reduction in step-count from one order to the next. In this language, the entire triangular differencing array associated with a finite originator forms a circuit. This viewpoint is particularly useful because the conjecture concerns a single distinguished boundary value at each stage, and the circuit language isolates that boundary value as a trace running through the entire differencing network. The framework is therefore not cosmetic: it is designed to make visible the finite combinatorics hidden in the repeated absolute differences and to reduce the conjecture to a statement about an invariant trace.\\

More concretely, the path formalism captures the local geometry of the differencing process. If a finite sequence has length $n$, then its first difference row has $n-1$ segments, the next has $n-2$, and so on, until the process ends. The paper systematically records this through the notions of order and step and shows that increasing the order decreases the maximal number of steps by exactly one. This elementary but important bookkeeping leads to the step--order relation and to the total count of segments appearing across all rows. The effect is to turn the recursive differencing scheme into a finite combinatorial diagram whose global size and local reduction can be uniformly controlled. That control is essential for later arguments because it allows one to quantify the accumulation of information across the whole differencing table rather than working row-by-row in isolation.\\

A second layer of structure is introduced through the length of a path. For a fixed row, the length is the sum of its segments, so it measures the total mass carried by that order of differencing. This quantity is not introduced for its own sake; it is a natural aggregate invariant that can be compared across adjacent orders. The paper develops upper and lower bounds for path lengths in terms of the previous row and uses them to show that the lengths decrease in a controlled way under suitable hypotheses. The inequalities are then transferred to the entire circuit, where the total circuit length is the sum of the lengths of all rows. In this aggregated setting, one obtains bounds for the full differencing network and, more importantly, a mechanism for passing from average information to the existence of specific small segments. The circuit length therefore functions as a coarse global measure from which finer structural information can be extracted.\\

The trace is the key invariant that connects the path formalism back to the Gilbreath conjecture. Fixing a segment position $s$ and summing the $s$th segment across all admissible orders produces the trace $\tau_{n,s}$. This quantity records how one chosen horizontal position propagates through the entire differencing circuit. Among all traces, the first trace $\tau_{n,1}$ is especially significant, because it precisely collects the leading entries of successive rows. In other words, the conjecture is not merely about the survival of a leading $1$ at each step; it is about the persistence of a whole trace of boundary data through the circuit. This is the point at which the framework becomes especially natural: the same recursive operation that defines the conjecture also generates the invariant whose value encodes the conjecture. The passage from the iterative arithmetic definition to the trace formalism therefore compresses the problem into a single structural quantity.\\

The main reduction developed in the paper is that the Gilbreath condition can be reformulated in terms of the circuit trace. In the finite setting, once the leading segments of all rows are positive and the first trace has the maximal possible value $n-1$, the leading segment in each row must be equal $1$. This is the decisive observation: the conjecture is converted from a statement about infinitely repeated absolute differences into a finite statement about the first trace of a circuit generated by an originator. The reduction clarifies what must ultimately be controlled. Rather than tracking every entry in every row, one studies whether the circuit admits the trace behavior that forces the boundary terms to remain minimal and nonzero. From this perspective, the conjecture becomes a problem about rigidity at the boundary of a finite iterated-difference network, and the notion of circuit provides the right combinatorial container for that rigidity.\\

The advantage of this reduction is conceptual as well as technical. Conceptually, the circuit formalism exposes the triangular recursion underlying the conjecture and makes the role of the first column transparent. Technically, it opens the door to inequalities involving lengths, traces, and boundary terms that can be combined and iterated. The developed framework thus supplies a systematic language in which local differencing information, global circuit size, and boundary persistence are all related. This is precisely the kind of structure one hopes for in a problem of this type: a conjecture about apparently simple arithmetic data is recast as a statement about a finite combinatorial geometry, where the obstruction to a proof may be studied through the interaction of paths, circuits, and traces.

\subsection{Organization of the paper} The paper is organized as follows. Section~2 introduces the notion of a path induced by a finite originator and develops the basic bookkeeping relations governing its order, steps, and segment structure. Section~3 studies the length of a path and establishes inequalities that compare a path with the one obtained in the next order, both for individual rows and for the aggregated circuit. Section~4 introduces the circuit itself and defines the trace of a fixed segment position across all rows; this section also derives the principal estimates relating the circuit length and trace. Section~5 gives the reduction of the Gilbreath conjecture to the circuit language and shows how the conjectural boundary behavior is encoded by the first trace. In this way, the paper builds from local differencing data to a global structural formulation designed to isolate the core mechanism behind the Gilbreath phenomenon.

\section{The notion of a path induced by a sequence}

In this section, we introduce and study the notion of a \emph{path} induced by an \emph{originator}.

\begin{definition}\label{path}
Let $\{a_i\}_{i=1}^{n}$ be any finite sequence. By the \emph{path} of \emph{order} $1$ with \emph{steps} $l\geq 1$ induced by the sequence, we mean the sequence $\{d_j^{1}\}_{j=1}^{l}$ such that 
\begin{align}
d_1^{1}=|a_2-a_1|, d_2^{1}=|a_3-a_2|,\ldots,d_l^{1}=|a_{l+1}-a_{l}|.\nonumber
\end{align}
Similarly, by the \emph{path} of \emph{order} $k\geq 2$ with $t$~($t<l$)~\emph{steps} induced by the sequence $\{a_i\}_{i=1}^{n}$, we mean the sequence $\{d_j^{k}\}_{j=1}^{t}$ such that 
\begin{align}
d_1^k=|d_2^{k-1}-d_1^{k-1}|,\ldots, d_t^{k}=|d_{t+1}^{k-1}-d_{t}^{k-1}|.\nonumber
\end{align}
We call each $d_{j}^{k}$ for $1\leq j\leq t$ a \emph{segment} of the path induced. We call $d_{1}^{k}$ the \emph{prime} segment of the path. We call the sequence $\{a_i\}_{i=1}^{n}$ the \emph{originator} of the paths. We denote by $a_i=d_i^{0}$ for $1\leq i\leq n$. Similarly, we call the originator the \emph{trivial} path induced with $\{a_i\}_{i=1}^{n}=\{d_i^0\}_{i=1}^{n}$.
\end{definition}
\bigskip

\begin{proposition}\label{step reduction}
Let $\{d_j^{k}\}_{j=1}^{t}$ be a path of order $k\geq 1$ with maximal step $t$ with originator $\{a_i\}_{i=1}^{n}$. The path $\{d_i^{k+1}\}_{i\geq 1}$ has exactly $t-1$ maximal steps.
\end{proposition}

\begin{proof}
Suppose that $\{d_j^{k}\}_{j=1}^{t}$ is a path of order $k\geq 1$ with maximal step $t$ and with originator $\{a_i\}_{i=1}^{n}$. We find that $d_{i}^{k+1}=|d_{j+1}^k-d_{j}^k|$ for $t-1 \geq j\geq 1$ is a segment of the path $\{d_i^{k+1}\}_{i\geq 1}$ and each of such a segment is \emph{uniquely} determined by the $t-1$ segments of the path $\{d_j^{k}\}_{j=1}^{t}$. It follows that the path $\{d_i^{k+1}\}_{i\geq 1}$ must have exactly $t-1$ maximal steps.
\end{proof}
\bigskip

The number of steps of paths induced by any sequence must experience some amount of drop with an increase in the order of the path. In particular, the number of steps in a path produced by some originator of order $l$ must be a unit more step than the path of order $l+1$ with the same originator.

\begin{proposition}\label{total steps}
Let $\{a_i\}_{i=1}^{n}$ be an originator of paths. The total number of maximal steps in all induced paths must be 
\begin{align}
\frac{n(n-1)}{2}.\nonumber
\end{align}
\end{proposition}

\begin{proof}
Suppose that $\{a_i\}_{i=1}^{n}$ is an originator of paths. By Proposition \ref{step reduction}, the path of order $1$ must have exactly $(n-1)$ maximal steps. The path of order $2$ must have exactly $(n-2)$ maximal steps. By induction, the path of order $k\geq 2$ must have $(n-k)$ maximal steps. Iterating downward, we generate the maximal steps of all such induced paths by the originator terminating at $1$. Thus, the total number of such maximal steps of all induced paths is
\begin{align}
1+2+\cdots+(n-2)+(n-1)&=\frac{n(n-1)}{2}.\nonumber
\end{align}
\end{proof}
\bigskip

We find that the step and order of a path are related to the number of terms in an originator. This is an easy consequence of Proposition \ref{total steps}.

\begin{proposition}[Step-order equation]\label{step-order equation}
Let $\{a_i\}_{i=1}^{n}$ be an originator of the path $\{d_j^{k}\}_{j=1}^{t}$. If the step is a maximal step, then 
$$
n=k+t.
$$
\end{proposition}

\begin{proof}
Suppose that $\{d_j^{k}\}_{j=1}^{t}$ is the path induced by the originator $\{a_i\}_{i=1}^{n}$. By Proposition \ref{total steps}, the number of maximal steps $t$ in the path must satisfy 
$$
t=n-k.
$$
\end{proof}

\section{The length of a path}

In this section, we introduce and study the notion of the \emph{length} of a path.

\begin{definition}\label{path length}
Let $\{d_j^{k}\}_{j=1}^{t}$ be a path of order $k\geq 1$ with step $t$ induced by the sequence $\{a_i\}_{i=1}^{n}$. By the \emph{length} of the path, denoted by $\iota_{t,k}$, we mean the finite sum
\begin{align}
\iota_{t,k}=\sum \limits_{j=1}^{t}d_j^{k}.\nonumber
\end{align}
\end{definition}
\bigskip


\begin{figure}[htbp]
\centering
\begin{tikzpicture}[
    font=\small,
    >=Latex,
    every node/.style={inner sep=2pt, outer sep=1pt},
    cell/.style={draw, rounded corners=2pt, minimum width=11mm, minimum height=7mm, align=center, fill=gray!6},
    origin/.style={draw, rounded corners=2pt, minimum width=11mm, minimum height=7mm, align=center, fill=orange!12, very thick},
    pathk/.style={draw, rounded corners=2pt, minimum width=11mm, minimum height=7mm, align=center, fill=blue!12, very thick, draw=blue!65!black},
    lab/.style={font=\bfseries, align=left}
]

\node[origin] (a1) at (0.0,0.0) {$a_1$};
\node[origin] (a2) at (1.35,0.0) {$a_2$};
\node[origin] (a3) at (2.70,0.0) {$\cdots$};
\node[origin] (a4) at (4.05,0.0) {$a_{n-1}$};
\node[origin] (a5) at (5.40,0.0) {$a_n$};

\node[lab, anchor=west] at (6.05,0.0) {originator};

\node[cell] (d11) at (0.70,-1.45) {$d_1^1$};
\node[cell] (d12) at (2.05,-1.45) {$d_2^1$};
\node[cell] (d13) at (3.40,-1.45) {$\cdots$};
\node[cell] (d14) at (4.75,-1.45) {$d_{n-1}^1$};

\node[lab, anchor=west] at (6.05,-1.45) {path of order $1$};

\node[cell] (d21) at (1.35,-2.90) {$d_1^2$};
\node[cell] (d22) at (2.70,-2.90) {$d_2^2$};
\node[cell] (d23) at (4.05,-2.90) {$\cdots$};

\node[lab, anchor=west] at (6.05,-2.90) {path of order $2$};

\node[pathk] (dk1) at (2.05,-4.35) {$d_1^k$};
\node[pathk] (dk2) at (3.40,-4.35) {$d_2^k$};
\node[pathk] (dk3) at (4.75,-4.35) {$\cdots$};
\node[pathk] (dks) at (6.10,-4.35) {$d_{m_k}^k$};

\node[lab, anchor=west] at (7.00,-4.35) {path of order $k$};

\node[cell] (d31) at (2.70,-5.80) {$\vdots$};
\node[cell] (d41) at (3.40,-7.25) {$d_1^{\,n-1}$};

\node[lab, anchor=west] at (6.05,-7.25) {path of order $n-1$};

\draw[-{Latex[length=3mm]}, very thick] (-1.10,0.25) -- (-1.10,-7.55);
\node[rotate=90, lab] at (-1.55,-3.60) {increasing order};

\draw[decorate, decoration={brace, amplitude=6pt}, very thick]
    (1.80,-4.75) -- (6.35,-4.75)
    node[midway, yshift=-12pt, align=center, font=\bfseries]
    {length of the path of order $k$};

\node[draw, rounded corners=2pt, fill=yellow!12, inner sep=4pt, align=left]
    at (9.15,-4.35)
    {$L_k$ is the horizontal sum of the terms\\[1mm]
     on the $k^{\text{th}}$ row};

\draw[-{Latex[length=2.5mm]}, thick] (8.20,-4.35) -- (6.35,-4.35);

\node[draw, rounded corners=2pt, fill=green!10, inner sep=5pt, align=left]
    at (9.05,-6.55)
    {$L_k
    = \displaystyle\sum_{j=1}^{m_k} d_j^k$};

\node[align=left] at (9.05,-7.40)
{horizontal sum along\\
the row of order $k$};

\draw[-{Latex[length=2.5mm]}, thick] (8.15,-6.55) -- (6.70,-4.35);

\end{tikzpicture}
\caption{Gilbreath phenomenon with the path of order \(k\) highlighted and its length interpreted as the horizontal sum of the terms on that row.}
\label{fig:gilbreath-path-length}
\end{figure}

The following inequality is crude, but it relates the length of each path to the worst segment of the previous consecutive path.

\begin{proposition}\label{length upper and lower bound}
Let $\{d_j^k\}_{j=1}^{t}$ be a path with originator $\{a_i\}_{i=1}^{n}$. For all $k\geq 1$, we have
\begin{align}
|d_{n-k+1}^{k-1}-d_1^{k-1}|\leq \iota_{n-k,k}\leq (n-k)\mathrm{max}\{|d_{j+1}^{k-1}-d_j^{k-1}|\}_{j=1}^{n-k}.\nonumber
\end{align}
\end{proposition}

\begin{proof}
By definition \ref{path length} and Proposition \ref{step-order equation}, we can write 
\begin{align}
\iota_{n-k,k}&=\sum \limits_{j=1}^{n-k}d_j^{k}\nonumber \\&=\sum \limits_{j=1}^{n-k}|d_{j+1}^{k-1}-d_j^{k-1}|\nonumber \\& \leq \mathrm{max}\{|d_{j+1}^{k-1}-d_j^{k-1}|\}_{j=1}^{n-k}\sum \limits_{j=1}^{n-k}1.\nonumber
\end{align}
This establishes the upper bound. The lower bound can be deduced by adding and deleting the segments of the path of order $(k-1)$ and subsequent use of the triangle inequality.
\end{proof}
\bigskip

Knowing the largest value of a segment in a given path provides information about at least one segment in the closest previous path. We apply the inequality devised in Proposition \ref{length upper and lower bound} to formalize this assertion.

\begin{proposition}\label{worst segment-previous segment relationship}
Let $\{d_j^k\}_{j=1}^{t}$ be a path with originator $\{a_i\}_{i=1}^{n}$. If $\mathrm{max}\{|d_{j+1}^{k-1}-d_j^{k-1}|\}_{j=1}^{n-k}\leq c$ for some $c>0$, then there exists at least some $1\leq m\leq (n-k)$ such that $d_m^k\leq c$.
\end{proposition}

\begin{proof}
Suppose that $\{d_j^k\}_{j=1}^{t}$ is a path with originator $\{a_i\}_{i=1}^{n}$. By Proposition \ref{length upper and lower bound}, we have
\begin{align}
\iota_{n-k,k}&\leq (n-k)\mathrm{max}\{|d_{j+1}^{k-1}-d_j^{k-1}|\}_{j=1}^{n-k}.\nonumber
\end{align}
Under the requirement $\mathrm{max}\{|d_{j+1}^{k-1}-d_j^{k-1}|\}_{j=1}^{n-k}\leq c$ for some $c>0$, we obtain
\begin{align}
\iota_{n-k,k}&\leq (n-k)\mathrm{max}\{|d_{j+1}^{k-1}-d_j^{k-1}|\}_{j=1}^{n-k}\nonumber \\&\leq c(n-k)\nonumber
\end{align}
so that the average value of the segments in the path with $(n-k)$ steps is 
\begin{align}
\frac{\iota_{n-k,k}}{(n-k)}&=\frac{1}{(n-k)}\sum \limits_{j=1}^{n-k}|d_{j+1}^{k-1}-d_j^{k-1}|\leq c.\nonumber
\end{align}
It implies that there must exist some $1\leq m\leq (n-k)$ such that $d_{m}^k\leq c$ for $c>0$. Suppose that for all such $1\leq m\leq (n-k)$ then $d_m^k>c$, we have $\iota_{n-k,k}>c(n-k)$. It follows that 
\begin{align}
c(n-k)<\iota_{n-k,k}\leq (n-k)\mathrm{max}\{|d_{j+1}^{k-1}-d_j^{k-1}|\}_{j=1}^{n-k}\nonumber
\end{align}
so that $c<\mathrm{max}\{|d_{j+1}^{k-1}-d_j^{k-1}|\}_{j=1}^{n-k}$, which is impossible.
\end{proof}

\begin{proposition}\label{decreasing length of path}
Let $\{d_j^{k}\}_{j=1}^{t}$ and $\{d_j^{k+1}\}_{j=1}^{t-1}$ be any two paths of the same originator such that $|d_{j+1}^{k}-d_j^{k}|\leq d_{j+1}^k$ for all $1\leq j\leq t$. We have
\begin{align}
\iota_{t-1,k+1}<\iota_{t,k}\nonumber
\end{align}
for all $k\geq 1$. 
\end{proposition}

\begin{proof}
By the definition \ref{path length}, we can write
\begin{align}
\iota_{t-1,k+1}&=\sum \limits_{j=1}^{t-1}d_j^{k+1}\nonumber \\&=\sum \limits_{j=1}^{t-1}|d_{j+1}^{k}-d_j^{k}|.\nonumber
\end{align}
Under the requirement $|d_{j+1}^{k}-d_j^{k}|\leq d_{j+1}^k$ for all $1\leq j\leq t$, we get
\begin{align}
\sum \limits_{j=1}^{t-1}|d_{j+1}^{k}-d_j^{k}|&\leq \sum \limits_{j=1}^{t-1}d_{j+1}^{k}\nonumber \\&=\sum \limits_{j=1}^{t}d_{j}^{k}=\iota_{t,k}.\nonumber
\end{align}
\end{proof}
\bigskip

The preceding development suggests that for all the paths induced by the originator $\{a_i\}_{i=1}^{n}$, the worst order and the least attainable steps are $n-1$ and $1$, respectively. We introduce and study the notion of a \emph{circuit}.

\section{The notion of a circuit}

In this section, we introduce and study the notion of a \emph{circuit} generated by paths induced by a certain originator.

\begin{definition}\label{circuit}
Let $\{a_i\}_{i=1}^{n}$ be a generator of the paths $\{d_{j}^{k}\}_{j\geq 1}$. We call the collection of all such paths for all $1\leq k\leq n-1$ the \emph{circuit} induced by the originator.
\end{definition}

\begin{definition}\label{circuit length}
Let $\{d_{j}^{k}\}_{j\geq 1}$ for all $1\leq k\leq n-1$ be the circuit induced by the originator $\{a_i\}_{i=1}^{n}$. We denote the length of the circuit by
\begin{align}
\kappa(n):=\sum \limits_{k=1}^{n-1}\iota_{n-k,k}.\nonumber
\end{align}
\end{definition}

\begin{proposition}\label{circuit weak bounds}
Let $\{d_{j}^{k}\}_{j\geq 1}$ for all $1\leq k\leq n-1$ be the circuit induced by the originator $\{a_i\}_{i=1}^{n}$. We have
\begin{align}
(n-1)\mathrm{min}\{|d_{n-k+1}^{k-1}-d_1^{k-1}|\}_{k=1}^{n-1}\leq \kappa(n)&\leq \sum \limits_{k=1}^{n-1}\mathrm{max}\{|d_{j+1}^{k-1}-d_j^{k-1}|\}_{j=1}^{n-k}\nonumber \\+\int \limits_{1}^{n-1}\bigg(\sum \limits_{s=1}^{t}\mathrm{max}\{|d_{j+1}^{s-1}-d_j^{s-1}|\}_{j=1}^{n-s}\bigg)dt. \nonumber
\end{align}
\end{proposition}

\begin{proof}
The lower bound can be deduced by applying the lower bound in Proposition \ref{length upper and lower bound}. The upper bound follows by an application of partial summation to the sum 
\begin{align}
\kappa(n):&=\sum \limits_{k=1}^{n-1}\iota_{n-k,k}\nonumber \\&\leq \sum \limits_{k=1}^{n-1}(n-k)\mathrm{max}\{|d_{j+1}^{k-1}-d_j^{k-1}|\}_{j=1}^{n-k}.\nonumber
\end{align}
\end{proof}

\begin{definition}\label{The trace of segments}
Let $\{d_{j}^{k}\}_{j\geq 1}$ for all $1\leq k\leq n-1$ be the circuit induced by the originator $\{a_i\}_{i=1}^{n}$. By the \emph{trace} of the $s^{th}$ segment of paths in a circuit, denoted by $\tau_{n,s}$, we mean the finite sum 
\begin{align}
\tau_{n,s}:=\sum \limits_{k=1}^{n-s}d_s^{k}.\nonumber
\end{align}
\end{definition}
\bigskip


\begin{figure}[htbp]
\centering
\begin{tikzpicture}[
    font=\small,
    >=Latex,
    every node/.style={inner sep=2pt, outer sep=1pt},
    cell/.style={draw, rounded corners=2pt, minimum width=11mm, minimum height=7mm, align=center, fill=gray!6},
    tracecell/.style={draw, rounded corners=2pt, minimum width=11mm, minimum height=7mm, align=center, fill=red!10, very thick, draw=red!70!black},
    ell/.style={font=\large},
    lab/.style={font=\bfseries}
]

\node[cell] (a1) at (0.0,0.0) {$a_1$};
\node[cell] (a2) at (1.3,0.0) {$a_2$};
\node[cell] (a3) at (2.6,0.0) {$\cdots$};
\node[tracecell] (as) at (4.0,0.0) {$a_s$};
\node[cell] (a4) at (5.4,0.0) {$\cdots$};
\node[cell] (an) at (6.8,0.0) {$a_n$};

\node[lab, anchor=west] at (7.35,0.0) {originator};

\node[cell] (d11) at (0.6,-1.45) {$d_1^{\,1}$};
\node[cell] (d12) at (1.9,-1.45) {$\cdots$};
\node[tracecell] (ds1) at (4.0,-1.45) {$d_s^{\,1}$};
\node[cell] (d13) at (5.2,-1.45) {$\cdots$};
\node[cell] (d1n) at (6.5,-1.45) {$d_{n-1}^{\,1}$};

\node[lab, anchor=west] at (7.35,-1.45) {order $1$};

\node[cell] (d21) at (1.2,-2.90) {$d_1^{\,2}$};
\node[cell] (d22) at (2.4,-2.90) {$\cdots$};
\node[tracecell] (ds2) at (4.0,-2.90) {$d_s^{\,2}$};
\node[cell] (d23) at (5.2,-2.90) {$\cdots$};
\node[cell] (d2n) at (6.2,-2.90) {$d_{n-2}^{\,2}$};

\node[lab, anchor=west] at (7.35,-2.90) {order $2$};

\node[cell] (d31) at (1.8,-4.35) {$\cdots$};
\node[tracecell] (ds3) at (4.0,-4.35) {$d_s^{\,3}$};
\node[cell] (d33) at (5.0,-4.35) {$\cdots$};

\node[lab, anchor=west] at (7.35,-4.35) {order $3$};

\node[tracecell] (ds4) at (4.0,-5.80) {$d_s^{\,4}$};

\node[lab, anchor=west] at (7.35,-5.80) {$\vdots$};

\node[tracecell] (dslast) at (4.0,-7.25) {$d_s^{\,n-s}$};

\node[lab, anchor=west] at (7.35,-7.25) {order $n-s$};

\draw[very thick, red!70!black, dashed]
    (as.south) -- (ds1.north);

\draw[very thick, red!70!black, dashed]
    (ds1.south) -- (ds2.north);

\draw[very thick, red!70!black, dashed]
    (ds2.south) -- (ds3.north);

\draw[very thick, red!70!black, dashed]
    (ds3.south) -- (ds4.north);

\draw[very thick, red!70!black, dashed]
    (ds4.south) -- (dslast.north);

\draw[decorate, decoration={brace, amplitude=6pt, mirror}, very thick]
    (4.55,-0.05) -- (4.55,-7.20)
    node[midway, xshift=1.55cm, align=left]
    {$\tau_{n,s}$\\[1mm]
     \textbf{trace of the $s^{\text{th}}$ segment}\\[1mm]
     $\displaystyle \tau_{n,s}=d_s^{\,1}+d_s^{\,2}+\cdots+d_s^{\,n-s}$};

\draw[-{Latex[length=3mm]}, very thick] (-1.25,0.25) -- (-1.25,-7.55);
\node[rotate=90, lab] at (-1.70,-3.65) {increasing order};

\node[align=left, font=\small] at (1.2,-8.25)
{The highlighted column consists of the $s^{\text{th}}$ segments.\\
Their sum is the trace $\tau_{n,s}$.};

\end{tikzpicture}
\caption{Gilbreath phenomenon with the trace of the $s^{\text{th}}$ segment marked as the sum of all its occurrences across the array.}
\label{fig:gilbreath-trace-s}
\end{figure}

\begin{proposition}\label{trace property}
Let $\{d_{j}^{k}\}_{j\geq 1}$ for all $1\leq k\leq n-1$ be the circuit induced by the originator $\{a_i\}_{i=1}^{n}$. We have  
$$
2\tau_{n,s}\geq (a_{s+1}-a_s)+d_s^{n-s}+\tau_{n,s+1}.
$$
\end{proposition}

\begin{proof}
We write 
\begin{align}
\tau_{n,s}:&=\sum \limits_{k=1}^{n-s}d_s^{k}\nonumber \\&=\sum \limits_{k=1}^{n-s}|d_{s+1}^{k-1}-d_s^{k-1}|\nonumber \\& \geq \sum \limits_{k=1}^{n-s}(d_{s+1}^{k-1}-d_s^{k-1})\nonumber \\&=\sum \limits_{k=1}^{n-s}d_{s+1}^{k-1}-\sum \limits_{k=1}^{n-s}d_s^{k-1}\nonumber \\&=\sum \limits_{i=0}^{n-s-1}d_{s+1}^{i}-\sum \limits_{i=0}^{n-s-1}d_s^{i}\nonumber \\&=d_{s+1}^{0}+\sum \limits_{i=1}^{n-(s+1)}d_{s+1}^{i}-\sum \limits_{i=1}^{n-s}d_s^{i}-d_s^{0}+d_s^{n-s}\nonumber \\&=(a_{s+1}-a_s)+d_s^{n-s}+\tau_{n,s+1}-\tau_{n,s}.\nonumber
\end{align}
This establishes the inequality.
\end{proof}

We can write the length of a circuit $\kappa(n)$ with the originator $\{a_i\}_{i=1}^{n}$ as the sum of the trace of segments of each kind within paths in the circuit. Hence, we obtain
\begin{align}
\kappa(n)&=\sum \limits_{k=1}^{n-1}\iota_{n-k,k}\nonumber \\&=\sum \limits_{k=1}^{n-1}\sum \limits_{s=1}^{n-k}d_s^k\nonumber
\end{align}
so that by interchanging the order of summation, we get
\begin{align}
\kappa(n)&=\sum \limits_{k=1}^{n-1}\sum \limits_{s=1}^{n-k}d_s^k\nonumber \\&=\sum \limits_{s=1}^{n-1}d_s^1+\sum \limits_{s=1}^{n-2}d_s^2+\cdots +\sum \limits_{s=1}^{n-(n-2)}d_s^{n-2}+d_s^{n-1}\nonumber \\&=\bigg(d_1^1+d_1^2+\cdots +d_1^{n-1}\bigg)+\bigg(d_2^1+d_2^2+\cdots+d_2^{n-2}\bigg)+\cdots+d_{n-1}^{1}\nonumber \\&=\sum \limits_{k=1}^{n-1}d_1^k+\sum \limits_{k=1}^{n-2}d_2^k+\cdots +\sum \limits_{k=1}^{n-(n-1)}d_{n-1}^k\nonumber \\&=\sum \limits_{s=1}^{n-1}\sum \limits_{k=1}^{n-s}d_s^k\nonumber \\&=\sum \limits_{s=1}^{n-1}\tau_{n,s}.\nonumber
\end{align}

This implies that the total length of any given circuit can also be obtained by summing the trace of each segment in a circuit. We can deduce an upper and lower bound for the average trace in a circuit using Proposition \ref{circuit weak bounds}:

\begin{proposition}\label{average trace bound}
Let $\{d_{j}^{k}\}_{j\geq 1}$ for all $1\leq k\leq n-1$ be the circuit induced by the originator $\{a_i\}_{i=1}^{n}$. We have
\begin{align}
(n-1)\mathrm{min}\{|d_{n-k+1}^{k-1}-d_1^{k-1}|\}_{k=1}^{n-1}\leq \sum \limits_{s=1}^{n-1}\tau_{n,s}&\leq (n-1)\mathrm{max}_{1\leq k\leq n-1}\mathrm{max}\{|d_{j+1}^{k-1}-d_j^{k-1}|\}_{j=1}^{n-k}\nonumber \\+\int \limits_{1}^{n-1}\bigg(\sum \limits_{s=1}^{t}\mathrm{max}\{|d_{j+1}^{s-1}-d_j^{s-1}|\}_{j=1}^{n-s}\bigg)dt. \nonumber
\end{align}
\end{proposition} 

\begin{proof}
The lower bound can be deduced from the lower bound in Proposition \ref{circuit weak bounds}. The upper bound follows by using the upper bound in Proposition \ref{circuit weak bounds} and noting that
\begin{align}
\sum \limits_{k=1}^{n-1}\mathrm{max}\{|d_{j+1}^{k-1}-d_j^{k-1}|\}_{j=1}^{n-k}&\leq (n-1) \mathrm{max}_{1\leq k\leq n}\mathrm{max}\{|d_{j+1}^{k-1}-d_j^{k-1}|\}_{j=1}^{n-k}.\nonumber
\end{align}
\end{proof}

The upper bound in Proposition \ref{average trace bound} suggests that on average the trace of segments in a circuit must be at most 
\begin{align}
\leq \mathrm{max}_{1\leq k\leq n}\mathrm{max}\{|d_{j+1}^{k-1}-d_j^{k-1}|\}_{j=1}^{n-k}\nonumber
\end{align}
so that there exists some $1\leq m\leq n-1$ such that $\tau_{n,m}\leq \mathrm{max}_{1\leq k\leq n}\mathrm{max}\{|d_{j+1}^{k-1}-d_j^{k-1}|\}_{j=1}^{n-k}$.  Here, we use the inequality in Proposition \ref{trace property} to establish an inequality that relates the length of a circuit to the terms of the originator and the trace of the first segment in each path in the circuit.

\begin{theorem}
Let $\{d_{j}^{k}\}_{j\geq 1}$ for all $1\leq k\leq n-1$ be the circuit induced by the originator $\{a_i\}_{i=1}^{n}$. We have
\begin{align}
\kappa(n)+\tau_{n,1}\geq (2a_n-a_{n-1}-a_1)+\sum \limits_{j=1}^{n-2}d_j^{n-j}.\nonumber
\end{align}
\end{theorem}

\begin{proof}
Iterating the inequality in Proposition \ref{trace property}, we obtain the following chains of inequalities
\begin{align}
2\tau_{n,1}\geq (a_2-a_1)+d_{1}^{n-1}+\tau_{n,2}\nonumber
\end{align}
\begin{align}
2\tau_{n,2}\geq (a_3-a_2)+d_2^{n-2}+\tau_{n,3}\nonumber
\end{align}
\begin{align}
\vdots \vdots \vdots \vdots \vdots \vdots \vdots \vdots \nonumber
\end{align}
\begin{align}
\vdots \vdots \vdots \vdots \vdots \vdots \vdots \vdots \nonumber
\end{align}
\begin{align}
2\tau_{n,n-2}\geq (a_{n-1}-a_{n-2})+d_{n-2}^2+\tau_{n,n-1}.\nonumber
\end{align}
Adding the left hand-sides and the right-hand sides of the chain, we obtain further the inequality 
\begin{align}
2\sum \limits_{s=1}^{n-2}\tau_{n,s}\geq (a_{n-1}-a_1)+\sum \limits_{j=1}^{n-2}d_j^{n-j}+\sum \limits_{s=2}^{n-1}\tau_{n,s}.\nonumber
\end{align}
By adding and deleting the term $2\tau_{n,n-1}$ on the left-hand side of the inequality and $\tau_{n,1}$ on the right-hand side, we obtain the refined inequality 
\begin{align}
2\sum \limits_{s=1}^{n-1}\tau_{n,s}\geq \sum \limits_{s=1}^{n-1}\tau_{n,s}+\sum \limits_{j=1}^{n-2}d_j^{n-j}+(a_{n-1}-a_1)+2\tau_{n,n-1}-\tau_{n,1}.\nonumber
\end{align}
We can write 
\begin{align}
\sum \limits_{s=1}^{n-1}\tau_{n,s}&\geq \sum \limits_{j=1}^{n-2}d_j^{n-j}+(a_{n-1}-a_1)+2\tau_{n,n-1}-\tau_{n,1}\nonumber \\&=(2a_n-a_{n-1}-a_1)+\sum \limits_{j=1}^{n-2}d_j^{n-j}-\tau_{n,1}\nonumber 
\end{align}
by exploiting the relation $\tau_{n,n-1}=d_{n-1}^1=a_n-a_{n-1}$. This proves the claimed inequality.
\end{proof}

\section{Reduction of the Gilbreath conjecture to the language of circuit}

In this section, we reduce the Gilbreath conjecture to the developed framework. We reformulate the problem in the language of the \emph{circuit}.

\begin{proposition}\label{zero existence}
Let $\{d_{j}^{k}\}_{j\geq 1}$ for all $1\leq k\leq n-1$ be the circuit induced by the originator $\{a_i\}_{i=1}^{n}$ with each $a_i\in \mathbb{Z}$. If $\tau_{n,s}<n-s$, then there exists at least some $t$ such that $d_s^t=0$ for $1\leq t\leq n-s$.
\end{proposition}

\begin{proof}
Under the assumption that $\{d_{j}^{k}\}_{j\geq 1}$ for all $1\leq k\leq n-1$ is the circuit induced by the originator $\{a_i\}_{i=1}^{n}$, we obtain the lower bound
\begin{align}
\tau_{n,s}&:=\sum \limits_{k=1}^{n-s}d_s^{k}\nonumber \\&\geq \mathrm{min}\{d_s^k\}_{k=1}^{n-s}\sum \limits_{k=1}^{n-s}1=(n-s)\mathrm{min}\{d_s^k\}_{k=1}^{n-s}\nonumber
\end{align}
and under the requirement $\tau_{n,s}<n-s$ with $\mathrm{min}\{d_s^k\}_{k=1}^{n-s}\in \mathbb{Z}^{+}\cup \{0\}$, we must take $\mathrm{min}\{d_s^k\}_{k=1}^{n-s}=0$. This completes the proof.
\end{proof}

\begin{proposition}\label{strong Gilbreath}
Let $\{d_{j}^{k}\}_{j\geq 1}$ for all $1\leq k\leq n-1$ be the circuit induced by the originator $\{a_i\}_{i=1}^{n}$. If $d_1^k>0$ for all $1\leq k\leq n-1$ and $\tau_{n,1}=n-1$ for all $n\geq 2$, then $d_1^k=1$ for all $1\leq k\leq n-1$.
\end{proposition}

\begin{proof}
Under the assumption that for the circuit $\{d_{j}^{k}\}_{j\geq 1}$ for all $1\leq k\leq n-1$ induced by the originator $\{a_i\}_{i=1}^{n}$ with $\tau_{n,1}=n-1$, it follows that 
\begin{align}
\tau_{n,1}&=\sum \limits_{k=1}^{n-1}d_{1}^k=n-1.\nonumber
\end{align}
Since there are $n-1$ prime segments in the sum and each prime segment $d_1^k>0$ for all $1\leq k\leq n-1$, we deduce $d_1^k=1$ for $1\leq k\leq n-1$.
\end{proof}
\bigskip


\begin{figure}[htbp]
\centering
\begin{tikzpicture}[
    font=\small,
    >=Latex,
    every node/.style={inner sep=2pt, outer sep=1pt},
    origin/.style={draw, rounded corners=2pt, fill=orange!12, very thick,
                   minimum width=11mm, minimum height=7mm, align=center},
    seg/.style={draw, rounded corners=2pt, fill=gray!6,
                minimum width=11mm, minimum height=7mm, align=center},
    prime/.style={draw, rounded corners=2pt, fill=blue!12, draw=blue!65!black, very thick,
                  minimum width=11mm, minimum height=7mm, align=center},
    trace/.style={draw, rounded corners=2pt, fill=green!12, draw=green!45!black, very thick,
                  minimum width=11mm, minimum height=7mm, align=center},
    lab/.style={font=\bfseries, align=left},
    tinylab/.style={font=\small, align=left}
]

\node[origin] (a1) at (0.0,0.0) {$a_1$};
\node[origin] (a2) at (1.7,0.0) {$a_2$};
\node[origin] (a3) at (3.4,0.0) {$\cdots$};
\node[origin] (a4) at (5.1,0.0) {$a_{n-1}$};
\node[origin] (a5) at (6.8,0.0) {$a_n$};

\node[prime] (d11) at (0.85,-1.45) {$d_1^{1}$};
\node[seg]   (d12) at (2.55,-1.45) {$d_2^{1}$};
\node[seg]   (d13) at (4.25,-1.45) {$\cdots$};
\node[seg]   (d14) at (5.95,-1.45) {$d_{n-1}^{1}$};

\node[prime] (d21) at (1.70,-2.90) {$d_1^{2}$};
\node[seg]   (d22) at (3.40,-2.90) {$\cdots$};
\node[seg]   (d23) at (5.10,-2.90) {$d_{n-2}^{2}$};

\node[prime] (d31) at (2.55,-4.35) {$d_1^{3}$};
\node[seg]   (d32) at (4.25,-4.35) {$\ddots$};

\node[prime] (dn1) at (3.40,-5.80) {$d_1^{\,n-1}$};

\node[lab] at (-2.35, 0.00) {order $0$};
\node[lab] at (-2.35,-1.45) {order $1$};
\node[lab] at (-2.35,-2.90) {order $2$};
\node[lab] at (-2.35,-4.35) {order $3$};
\node[lab] at (-2.35,-5.80) {order $n-1$};

\draw[-{Latex[length=3mm]}, very thick] (-3.00,0.35) -- (-3.00,-6.15);
\node[rotate=90, font=\bfseries] at (-3.45,-2.95) {increasing order};

\draw[decorate, decoration={brace, amplitude=5pt}, very thick]
    (-0.55,0.55) -- (7.35,0.55)
    node[midway, yshift=11pt, font=\bfseries]
    {originator};

\draw[decorate, decoration={brace, amplitude=5pt}, thick]
    (0.30,-1.95) -- (6.30,-1.95)
    node[midway, yshift=10pt, font=\bfseries]
    {path of order $1$};

\draw[decorate, decoration={brace, amplitude=5pt}, thick]
    (1.15,-3.40) -- (5.60,-3.40)
    node[midway, yshift=10pt, font=\bfseries]
    {path of order $2$};

\draw[decorate, decoration={brace, amplitude=5pt}, thick]
    (2.00,-4.85) -- (4.95,-4.85)
    node[midway, yshift=10pt, font=\bfseries]
    {path of order $3$};

\draw[decorate, decoration={brace, amplitude=5pt}, thick]
    (3.00,-6.30) -- (3.80,-6.30)
    node[midway, yshift=10pt, font=\bfseries]
    {path of order $n-1$};

\draw[very thick, dashed, green!50!black] (0.85,0.35) -- (0.85,-6.05);
\draw[very thick, dashed, green!50!black] (1.70,0.35) -- (1.70,-6.05);
\draw[very thick, dashed, green!50!black] (2.55,0.35) -- (2.55,-6.05);
\draw[very thick, dashed, green!50!black] (3.40,0.35) -- (3.40,-6.05);

\node[trace] (tr) at (10.15,-2.85)
{$\tau_{n,1}=d_1^{1}+d_1^{2}+\cdots+d_1^{\,n-1}$};

\node[tinylab, align=left] at (10.15,-1.85)
{trace\\[1mm] of the first segment};

\draw[-{Latex[length=2.5mm]}, thick] (9.00,-2.55) -- (4.00,-2.55);

\draw[decorate, decoration={brace, amplitude=6pt, mirror}, very thick]
    (7.55,0.65) -- (7.55,-6.15)
    node[midway, xshift=1.45cm, font=\bfseries, align=left]
    {$\mathfrak{C}$\\[1mm]
     circuit\\[1mm]
     \small(the full triangular array)};

\node[tinylab] at (9.95,-5.55) {prime segment};
\draw[-{Latex[length=2.5mm]}, thick] (9.45,-5.55) -- (4.10,-5.80);

\node[tinylab] at (9.95,-4.45) {first column};
\draw[-{Latex[length=2.5mm]}, thick] (9.35,-4.45) -- (3.40,-4.35);

\node[draw, rounded corners=2pt, fill=yellow!10, inner sep=4pt, align=left]
at (10.15,-6.55)
{Gilbreath phenomenon:\\
the first entry in each path is the key statistic};

\end{tikzpicture}
\caption{Upside-down triangular representation of the Gilbreath framework: originator, paths of successive order, prime segments, trace, and circuit.}
\label{fig:gilbreath-clean}
\end{figure}

We can restate the Gilbreath conjecture in the language of the trace. This would imply that proving this version of the conjecture would imply the actual version of the conjecture.

\begin{conjecture}[Gilbreath]
Let $\mathbb{P}$ denotes the set of all prime numbers and $\{d_{j}^{k}\}_{j\geq 1}$ for all $1\leq k\leq n-1$ be the circuit induced by the originator $\{p_i\}_{i=1}^{n}$ where each $p_i\in \mathbb{P}$. We have $d_1^k>0$ for all $1\leq k\leq n-1$ and $\tau_{n,1}=n-1$ for all $n\geq 2$.
\end{conjecture}

\footnote{
\par
.}%
.

\bibliographystyle{amsplain}

\end{document}